\documentclass[a4paper,11pt,onecolumn]{amsart}
\usepackage{amsfonts}
\usepackage{pb-diagram}
\title[Asymptotic linearity of the mapping class group]{Asymptotic linearity of the mapping class group and a homological version of the Nielsen-Thurston classification}

\newtheorem{thm}{Theorem}

\newtheorem{lemma}[thm]{Lemma}

\newtheorem{cor}[thm]{Corollary}
\newtheorem{prop}[thm]{Proposition}

\newtheorem{quest}[thm]{Question}

\numberwithin{thm}{section}

\newcommand{\al}{\alpha}
\newcommand{\eps}{\epsilon}
\newcommand{\gam}{\Gamma}

\newcommand{\bZ}{\mathbb{Z}}
\newcommand{\bR}{\mathbb{R}}
\newcommand{\bC}{\mathbb{C}}

\DeclareMathOperator{\Aut}{Aut}
\DeclareMathOperator{\Out}{Out}
\DeclareMathOperator{\Homeo}{Homeo}

\DeclareMathOperator{\Mod}{Mod}

\DeclareMathOperator{\rk}{rk}

\newcommand{\bQ}{\mathbb{Q}}

\newcommand{\whG}{\widehat{G}}

\author[T. Koberda]{Thomas Koberda}
\address{Department of Mathematics\\ Harvard University\\ 1 Oxford St.\\ Cambridge, MA 02138 }
\email{ koberda@math.harvard.edu}
\subjclass{Primary 57M10; Secondary 57M99}
\keywords{Mapping class groups, braid groups, Nielsen-Thurston classification, representations of the mapping class group, automorphisms of free groups}
\begin{document}
\begin{abstract}
We study the action of the mapping class group on the real homology of finite covers of a topological surface.  We use the homological representation of the mapping class to construct a faithful infinite-dimensional representation of the mapping class group.  We show that this representation detects the Nielsen-Thurston classification of each mapping class.  We then discuss some examples that occur in the theory of braid groups and develop an analogous theory for automorphisms of free groups.  We close with some open problems.
\end{abstract}
\maketitle
\begin{center}
\today
\end{center}
\tableofcontents
\section{Introduction, motivation and statement of results}
The mapping class group of a surface $\Sigma$ is defined to be the group of self-homeomorphisms up to isotopy, and is usually denoted $\Mod(\Sigma)$.  We will often consider $\Sigma$ with a marked point which must be preserved by the isotopies, in which case the associated mapping class group is denoted $\Mod^1(\Sigma)$.  A major unsolved problem in the theory of mapping class groups is whether or not they are linear (cf. \cite{F}).  It is known that if $\Sigma$ is closed of genus $2$ then $\Mod(\Sigma)$ is linear by \cite{BB2}.

Some care must be taken when $\Sigma$ has punctures or boundary components in order to give a proper definition of the mapping class group.  When $\Sigma$ is an open disk with $n$ marked points (sometimes thought of as punctures), $\Mod(\Sigma)$ is usually defined to be the group of compactly supported homeomorphisms up to compactly supported isotopy and is often denoted by $B_n$, the braid group on $n$ strands.  Braid groups are known to be linear by \cite{B}.

There is an analogous problem for $\Aut(F_n)$, the automorphism group of a free group of finite rank.  $\Aut(F_n)$ contains a natural copy of $B_n$, but it is known that $\Aut(F_n)$ is not linear by \cite{FP}.

It is therefore natural to look at representations of the mapping class group, especially ones which arise in topological and geometric contexts, and analyze their faithfulness.  One obvious candidate is the representation of $\Mod(\Sigma)$ as a group of automorphisms of $H_1(\Sigma,\bZ)$.  We obtain a representation \[\Mod(\Sigma)\to Sp_{2g}(\bZ)\] when $\Sigma$ is closed of genus $g$, and otherwise \[\Mod(\Sigma)\to GL_n(\bZ)\] where $n=\rk H_1(\Sigma,\bZ)$.  This representation is given by taking a mapping class $\psi$, lifting it to a homeomorphism of $\Sigma$, and then looking at its action on the homology of $\Sigma$.  Therefore, one would be justified in writing the representation as \[\Mod(\Sigma)\to \Aut(H_1(\Sigma,\bZ)).\]

The homology representation of the mapping class group carries a great deal of information about mapping classes.  For instance, no finite order automorphism is contained in the kernel of the homology representation.  This is a consequence of the standard fact from the theory of mapping class groups (see \cite{FM}, for instance) that a finite order automorphism must fix a complex structure on $\Sigma$, and it is standard from algebraic geometry that holomorphic automorphisms of an algebraic curve act nontrivially on the homology of the curve.

Unfortunately, the homology representation also forgets a lot of information.  It is well-known that the homology representation of the mapping class group has a very large kernel, called the {\bf Torelli group} and is often denoted $\mathcal{I}(\Sigma)$.  The Torelli group is very complicated -- to study it, one often considers a filtration on $\mathcal{I}(\Sigma)$ called the {\bf Johnson filtration} (see \cite{J} and the references therein).  The Johnson filtration is defined in terms of the lower central series $\{\gamma_i(\pi_1(\Sigma))\}$.  We shall assume that the Johnson filtration lies inside of $\Mod^1(\Sigma)$ in order to avoid technical difficulties involved in giving proper definitions.  Recall that $\gamma_0(\pi_1(\Sigma))=\pi_1(\Sigma)$ and $\gamma_i(\pi_1(\Sigma))=[\pi_1(\Sigma),\gamma_{i-1}(\pi_1(\Sigma))]$.  The $k^{th}$ term $J_k$ of the Johnson filtration is the kernel of the natural map \[\Mod^1(\Sigma)\to\Aut(\pi_1(\Sigma)/\gamma_k(\pi_1(\Sigma)).\]  Note that $J_1=\mathcal{I}(\Sigma)$.  Johnson established that when $\Sigma$ is not a thrice- or four-times-punctured sphere or a once-punctured torus, $J_k$ is nontrivial for all $k$ and does not consist entirely of inner automorphisms.  The fact that \[\bigcap_k J_k=\{1\}\] is a consequence of the fact that \[\bigcap_k\gamma_k(\pi_1(\Sigma))=\{1\},\] a fact which is originally due to Magnus (see \cite{MKS} for a classical treatment).

The goal of this paper is to show that if one is willing to consider the actions of mapping classes on the homology of finite covers of $\Sigma$, one can recover much of the information that the homological representation forgets.  This is where the ``asymptotic" part of the title comes from: we see more and more nontrivial mapping classes acting nontrivially on the homology of covers as we look at higher and higher covers of $\Sigma$.  Asymptotic properties, especially asymptotic linearity, have been studied by many authors: see \cite{A}, for example.  Before we state the main results, we fix some notation.

Let $\Sigma=\Sigma_{g,n}$ be a connected hyperbolic type surface with genus $g$ and $n$ punctures.  By hyperbolic type, we mean that $\Sigma$ admits a finite volume complete hyperbolic metric, though we will not use any aspects of hyperbolic geometry in the main part of the paper.  All surfaces we consider will be orientable but we will not appeal to any fixed orientation.  Fix a basepoint $*\in\Sigma^0$, the interior of $\Sigma$.  Throughout we will denote the fundamental group of $\Sigma$ by $\pi_1(\Sigma,*)$, though we will suppress the basepoint in the notation.  Recall that the marked mapping class group of $\Sigma$ is defined as \[\Mod^1(\Sigma)=\pi_0(\Homeo^+(\Sigma),*),\] the group of orientation preserving homeomorphisms of $\Sigma$ up to  isotopy, along with a distinguished marked fixed during the isotopy and by each homeomorphism.  We do not necessarily puncture the surface at the marked point.  We identify homotopy classes of curves based at $*$ with elements of $\pi_1(\Sigma,*)$ and the free homotopy classes of essential closed curves in $\Sigma$ with conjugacy classes in $\pi_1(\Sigma)$.  This allows us to identify $\Mod^1(\Sigma)$ with a subgroup of $\Aut(\pi_1(\Sigma))$ and $\Mod(\Sigma)$ with a subgroup of $\Out(\pi_1(\Sigma))$.  We perform this identification once and for all, so that the lifts of mapping classes to covers are unambiguous as far as their actions on homology are concerned.  A note on terminology: if $\Sigma'\to\Sigma$ is a regular cover and $\pi_1(\Sigma')<\pi_1(\Sigma)$ is a characteristic subgroup, we say that the cover is a {\bf characteristic cover}.  We call the quotient of a group by a characteristic subgroup a {\bf characteristic quotient}.

Our first main result is:

\begin{thm}\label{t:nontrivial}
Let $\psi\in\Mod^1(\Sigma)$ and let $\gam$ be a finite characteristic quotient of $\pi_1(\Sigma)$ and let $\Sigma_{\gam}$ denote the associated covering space of $\Sigma$.  If $1\neq\psi\in\Aut(\gam)$ then $\psi$ acts nontrivially on $H_1(\Sigma_{\gam},\bZ)$.
\end{thm}

The fundamental group of a surface admits many characteristic quotients.  We can take the deck group to be solvable, nilpotent, or even a $p$-group.  When we decorate a cover with an adjective such as solvable, nilpotent, etc. we mean that the deck group has this property.  The fundamental observation is that if we take a sequence of quotients which exhaust $\pi_1(\Sigma)$, then each automorphism of $\pi_1(\Sigma)$ acts nontrivially on one of the quotients.  We immediately obtain:
\begin{cor}
Let $\{\Sigma_i\}$ be a sequence of exhausting, finite characteristic covers of $\Sigma$ and $\psi\in\Mod^1(\Sigma)$.  Then $\psi$ induces a nontrivial automorphism of $H_1(\Sigma_i,\bZ)$ for some $i$.  In particular, we may assume $\psi$ acts nontrivially on the homology of a solvable, nilpotent or even $p$-cover for any prime $p$.
\end{cor}

Let $\Sigma'\to\Sigma$ be a characteristic cover with deck group $\gam$.  It is not immediately clear that the action of $\psi$ on $H_1(\Sigma',\bZ)$ will not coincide with the action of some element of $\gam$.  To this end, we have the following extension of Theorem \ref{t:nontrivial}:
\begin{thm}\label{t:inner}
Let $\psi\in \Mod^1(\Sigma)$.  Suppose that for every finite characteristic cover $\Sigma'\to\Sigma$ the action of $\psi$ on $H_1(\Sigma',\bZ)$ coincides with that of an element of the deck group.  Then $\psi$ induces an inner automorphism of $\pi_1(\Sigma)$.
\end{thm}

Theorem \ref{t:nontrivial}, remains true for arbitrary automorphisms of $\pi_1(\Sigma)$ regardless of whether or not they are actually induced by homeomorphisms of $\Sigma$.  We will generally not distinguish notationally between a homeomorphism and its isotopy class.  We will consistently appeal to the fact that the map $\Mod^1(\Sigma)\to\Aut(\pi_1(\Sigma))$ is injective.  We remark briefly that if we forget $*\in\Sigma$ during isotopies, we get the standard mapping class group $\Mod(\Sigma)$.  Since changing the basepoint for an automorphism $\psi$ of $\pi_1(\Sigma)$ is tantamount to replacing $\psi$ by a conjugate, we obtain an injective map $\Mod(\Sigma)\to\Out(\pi_1(\Sigma))$.  It is well-known that the map $\Mod^1(\Sigma)\to\Mod(\Sigma)$ does not split (cf. \cite{Bi}).

Theorem \ref{t:nontrivial} can be viewed as positive evidence towards a question which McMullen asked the author.  McMullen has since answered the question in the strongest negative sense possible in \cite{Mc}.  To state the question properly, we develop some notation.  Let $\psi\in\Mod(\Sigma)$ be a pseudo-Anosov mapping class and let $K_{\psi}=K$ be its {\bf geometric dilatation}.  A reader unfamiliar with pseudo-Anosov homeomorphisms should consult the definitions given below.  The dilatation of the pseudo-Anosov map $\psi$ is the exponential of the entropy for the least-entropy representative of $\psi$ in its isotopy class.  Consider the collection of $\psi$-invariant finite covers of $\Sigma$, $\{\Sigma'\}\to\Sigma$.  We fix a lift of $\psi$ to each $\Sigma'$, and we let $K_H(\Sigma')$ be the {\bf homological dilatation} of $\psi$, namely its spectral radius as an automorphism of $H_1(\Sigma',\bR)$.  Since $K$ and $K_H$ are can be defined in terms of word growth in groups, it follows that $K$ is constant over this family and both $K_H$ and $K$ are independent of the choice of lift (cf. \cite{FLP}).  McMullen's question can be stated as:

\begin{quest}\label{q:homdil}
Is \[\sup_{\Sigma'\to\Sigma}K_H(\Sigma')=K?\]
\end{quest}

The original motivation for this question was the work of Friedl and Vidussi in \cite{FV}, which shows that the fibering of $3$-manifolds is detected by twisted Alexander polynomials (see \cite{FK} for more details about the twisted Alexander polynomial).  Since twisted Alexander polynomials encode a large amount of data about finite covers, it seems that finite covers of a surface should provide a wealth of information about mapping classes.  He was also motivated by the work of Kazhdan in \cite{Ka}, also \cite{Rh}, where it is shown that a hyperbolic metric on a Riemann surface can be recovered from the Jacobians of its finite covers.

McMullen proves:
\begin{thm}
Suppose that $\psi$ is a pseudo-Anosov homeomorphism of a surface with dilatation $K$.  Then either $K$ is the spectral radius of the action of $\psi$ on a finite cover of $\Sigma$, or there is an $0\leq\al<1$ such that \[\sup_{\Sigma'\to\Sigma}K_H(\Sigma')=\al K.\]
\end{thm}

Even though Question \ref{q:homdil} is completely answered, one can still attempt to understand the action of $\Mod^1(\Sigma)$ on the homology of finite covers.  We will see later that at least \[\sup_{\Sigma'\to\Sigma}K_H(\Sigma')\leq K,\] so that the action of a mapping class on the virtual homology of $\Sigma$ is a reasonable lower bound for the value of its dilatation.  This fact can also be seen in \cite{Roy2}, but we will develop a theory which is more general and is divorced from the fact that $\psi$ is a homeomorphism of an actual surface.  We will prove:
\begin{thm}\label{t:bound}
Let $M$ be a compact manifold equipped with a metric which is compatible with the manifold topology, and let $\psi$ be a $K$--Lipschitz  homeomorphism of $M$.  Let $K_{H,i}$ denote the homological dilatation of $\psi$ on the $i^{th}$ homology with real coefficients.  Then $K_{H,i}\leq K^i$.
\end{thm}

The standard homology representation of the mapping class group carries a great deal of information regarding the action of homeomorphisms on simple closed curves on $\Sigma$.  Recall the {\bf Nielsen-Thurston classification} of mapping classes (cf. \cite{CB}, \cite{FLP}, \cite{FM} for instance).  It says that each mapping class $\psi\in\Mod(\Sigma)$ can be classified by its action on the set of simple closed curves on a surface: a mapping class $\psi$ is called {\bf finite order} if has finite order in $\Mod(\Sigma)$.  It is a nontrivial fact that then it has a representative which is a finite order homeomorphism of $\Sigma$.  A mapping class $\psi$ is called {\bf reducible} if there is a finite collection $\mathcal{C}$ of disjoint, distinct, nonperipheral isotopy classes of simple closed curves which is preserved by $\psi$.  An Euler characteristic argument shows that the order of $\mathcal{C}$ can be bounded in terms of the topology of $\Sigma$.  In particular, some power of $\psi$ fixes the isotopy class of a nonperipheral simple closed curve.  A mapping class $\psi$ is called {\bf pseudo-Anosov} if it has infinite order and is not reducible.

One of the main results of this paper is that the homology of all finite covers of $\Sigma$ is sufficiently rich to reveal the Nielsen-Thurston of each mapping class.  Before stating the theorems precisely, we will give some more background and setup.

The homology representation of $\Mod(\Sigma)$ can also determine whether certain mapping classes are pseudo-Anosov.  Recall that there is a well-known {\bf Casson-Bleiler criterion} which can certify that certain mapping classes are pseudo-Anosov (see \cite{CB}, also \cite{Ma}).  To use the criterion, one take the image of $\psi$ under the homological representation computes its characteristic polynomial $p_{\psi}(t)$.  The criterion asserts:

\begin{prop}
Suppose that $p_{\psi}(t)$ is irreducible over $\bQ$, none of its zeros are roots of unity, and that that $p_{\psi}(t)$ is not a polynomial in $t^n$ for any $n>1$.  Then $\psi$ is a pseudo-Anosov mapping class.
\end{prop}
We will prove in a precise sense that no na\"ive generalization of the Casson-Bleiler criterion could possibly hold.
The Casson-Bleiler criterion will not detect all pseudo-Anosov homeomorphisms, as it is known that that $\mathcal{I}(\Sigma)$ contains pseudo-Anosov mapping classes.  The fact that $\mathcal{I}(\Sigma)$ contains pseudo-Anosov classes is a reflection of a more general fact due to Ivanov (see \cite{I}):
\begin{prop}\label{p:ivanov}
Let $\Sigma$ be a hyperbolic type surface which is not the thrice-punctured sphere.  Suppose $H<\Mod(\Sigma)$ is non-central and normal (the former condition is unnecessary when the genus of $\Sigma$ is greater than $2$).  Then every coset of $H$ in $\Mod(\Sigma)$ contains a pseudo-Anosov mapping class.
\end{prop}

In particular, it follows not only that $\mathcal{I}(\Sigma)$ and each term of the Johnson filtration contains pseudo-Anosov elements, but that one can fix any symplectic matrix and find pseudo-Anosov mapping classes which induce that automorphism on the homology of $\Sigma$.

The discussion above shows that the usual homological representation of the mapping class group forgets a lot of data about a mapping class.  The point of this paper is that a large part of the lost data can be recovered if we allow ourselves to consider a virtual homological representation.  Before we state the remaining main results, we say a few more general things.

Let $G$ be a group $\psi\in\Aut(G)$ and let $G'<G$ be $\psi$-invariant.  Then $\psi$ restricts to an automorphism of $G'$.  When $G'$ is characteristic (in which we call the cover given by the quotient $G\to G/G'$ a characteristic cover), we get a restriction map $\Aut(G)\to\Aut(G')$, and it is well-known that this map is in general neither injective nor surjective.  When $G'$ has finite index in $\pi_1(\Sigma)$, the restriction map is injective, at least when restricted to $\Mod^1(\Sigma)$.  Indeed, clearly the map is a homeomorphism, and by \cite{I} it is enough to show that no pseudo-Anosov homeomorphism is contained in the kernel.  A pseudo-Anosov homeomorphism stabilizes a measured foliation which lifts to every finite cover (see \cite{FLP}) and is therefore characterized by its local geometry.  This implies that the lift of $\psi$ is pseudo-Anosov and hence nontrivial.

Theorem \ref{t:nontrivial} and the homological Nielsen-Thurston classification can be conveniently stated in terms of a certain representation.  Let $F$ be a field which we assume to have characteristic zero in order to avoid technical difficulties.  Let $\Sigma'$ be a finite characteristic cover of $\Sigma$.  We suppose that $\Sigma'$ is an element in a class $\mathcal{K}$ of finite characteristic covers of $\Sigma$.  We may define the {\bf pro-$\mathcal{K}$ $F$-homology} of $\Sigma$ by taking
\[
H(\Sigma)=\varprojlim H_1(\Sigma',F),
\]
as $\Sigma'$ ranges over all elements of $\mathcal{K}$.  If $P$ denotes the set of punctures of $\Sigma'$, we can define the relative pro-$\mathcal{K}$ homology of $\Sigma$ by \[H_R(\Sigma)=\varprojlim H_1(\Sigma',P,F).\]  The field of coefficients and $\mathcal{K}$ will generally be apparent from context and we will suppress them from the notation.  We will not exploit any abstract properties of $H(\Sigma)$ other than the fact that it is a vector space of infinite dimension when the characteristic of $F$ is zero.  In this context, Theorem \ref{t:nontrivial} can be restated as saying that if $\mathcal{K}$ exhausts $G$ then $H(\Sigma)$ is a faithful representation of $\Aut(G)$.

Note that $H(\Sigma)$ comes equipped with a natural action of $\whG$, the pro-$\mathcal{K}$ completion of $\pi_1(\Sigma)$.  It is clear from the definitions that this action is compatible with the action of $\Aut(\pi_1(\Sigma))$.  Indeed, if $\Sigma'\to\Sigma$ is a finite regular cover with deck group $\gam$ then $\gam$ acts on itself and on $\pi_1(\Sigma')$ by conjugation.  The action descends to the abelianization of $\pi_1(\Sigma')$.  If \[\Sigma''\to\Sigma'\to\Sigma\] is a characteristic tower of covers with the total deck group given by $\gam$, then the conjugation action of $\gam$ on $\pi_1(\Sigma')$ factors through the quotient map from $\gam$ to the deck group of $\Sigma'/\Sigma$.  Thus, the action of the pro-$\mathcal{K}$ completion of $G$ acts on all $\mathcal{K}$ covers in a way which is compatible with the covering maps.  If $\psi$ is an automorphism of $G$, then $\psi$ acts on each $\mathcal{K}$-cover since these covers are all characteristic.  The action of $\whG$ is given by conjugation within $\pi_1(\Sigma)$, so the action of $\psi$ is compatible with the action of $\whG$.

By reversing the algebraic limit in the definition of $H(\Sigma)$, we may define the {\bf pro-$\mathcal{K}$ $F$-cohomology} of $\Sigma$.  In general it will be clear if we mean pro-$\mathcal{K}$ homology or cohomology, so we will not distinguish between these notationally.

The other main result of this paper is that the representation $H(\Sigma)$ can detect the Nielsen-Thurston classification of mapping classes.  Since the Nielsen-Thurston classification is characterized by actions of homeomorphisms on isotopy classes of curves, one way to proceed is to homologically encode closed curves on $\Sigma$ inside of $H(\Sigma)$.  We will do this by defining certain vectors of {\bf finite type}, whose $\pi_1(\Sigma)$-orbits are in bijective correspondence with free homotopy classes in $\pi_1(\Sigma)$.  With this terminology, we can state the next theorem:

\begin{thm}\label{t:nt}
Let $\mathcal{K}$ be any exhausting class of finite solvable characteristic covers of $\Sigma$ and let $H(\Sigma)$ be the pro-$\mathcal{K}$ rational cohomology of $\Sigma$.  Let $\psi\in\Mod(\Sigma)$, and choose a lift of $\psi$ to $\Mod^1(\Sigma)$.  The action of $\psi$ on $H(\Sigma)$ determines the Nielsen-Thurston classification of $\psi$ as follows:
\begin{enumerate}
\item
$\psi$ has finite order if and only if some power of $\psi$ preserves the $\pi_1(\Sigma)$-orbits of all finite type vectors in $H(\Sigma)$.
\item
$\psi$ is reducible if and only if some power of $\psi$ fixes the $\pi_1(\Sigma)$-orbit of some non-peripheral finite type vector in $H(\Sigma)$.
\item
$\psi$ is pseudo-Anosov if and only if the previous two conditions fail.
\end{enumerate}
The powers in items $(1)$ and $(2)$ are the order of $\psi$ and a power of $\psi$ which fixes the isotopy class of an essential simple closed curve, respectively.
\end{thm}

Unfortunately, the finite-type vectors do not give much insight into the action of a mapping class on the homology or cohomology of the cover, since the finite type vectors do not give rise to a subrepresentation of $H(\Sigma)$, but we shall show in a precise sense that they are the only way to encode homotopy classes of loops in $\Sigma$ as vectors in $H(\Sigma)$.  There is another homological version of the Nielsen-Thurston classification.  To state it, we need some setup which applies to general faithful representations of $\Mod(\Sigma)$:

Let $\gam$ be a characteristic quotient of $\pi_1(\Sigma)$.  If $\rho$ is a finite-dimensional representation of $\gam$ in characteristic zero, or more generally a direct limit of finite-dimensional representations in characteristic zero, then at any finite stage we may decompose the representation as a direct sum of irreducible modules.  If $\rho$ is an irreducible representation of $\gam$ and $\al\in\Aut(\gam)$, we get a new representation of $\gam$ via $\rho\circ\al$.  To each representation $\rho$ we can associate its character $\chi$.  $\Aut(\gam)$ acts on the characters of $\gam$ by precomposition.  We say two characters $\chi$ and $\chi'$ are {\bf equivalent} on a cyclic subgroup $\langle g\rangle<\gam$ if $\chi$ and $\chi'$ are equal as functions on $\langle g\rangle$, and inequivalent otherwise.

\begin{thm}\label{t:nt2}
Let $H(\Sigma)$ be the pro-$\mathcal{K}$ complex cohomology, where $\mathcal{K}$ is the class of all characteristic finite covers of $\Sigma$.  The action of $\psi\in\Mod(\Sigma)$ on the finite representations of $\whG$ in $H(\Sigma)$ determines the Nielsen-Thurston classification of $\psi$ as follows:
\begin{enumerate}
\item
$\psi$ is finite order if and only if some power of $\psi$ induces the trivial automorphism of the representations of the deck group for every $\mathcal{K}$--cover of $\Sigma$.
\item
$\psi$ is reducible if and only if there is a power of $\psi$ and a $1\neq g\in \pi_1(\Sigma)$ such that on every finite $\mathcal{K}$--cover of $\Sigma$ and every representation of the deck group, $\chi$ and $\chi\circ\psi$ are equivalent on the image of $\langle g\rangle$.
\item
$\psi$ is pseudo-Anosov if and only if the previous two conditions fail.
\end{enumerate}
\end{thm}

Note that even though the conditions use the action of $\Mod^1(\Sigma)$ on certain quotients of $\pi_1(\Sigma)$, the characterizations of the mapping classes does actually exploit the action of automorphisms on the homology of finite covers.  This is because every representation of the deck group of a finite cover of $\Sigma$ occurs as a summand of the homology of the cover.  The homological Nielsen-Thurston classification uses the fact that certain mapping classes permute the representations of the deck groups in a certain way.

We will formulate and prove an analogous result for automorphisms of free groups in Section \ref{s:free}.  

The terminology of the Nielsen-Thurston classification is reminiscent of representation-theoretic terminology, and it is interesting to see what the relationship between topology and representation theory is.  What Theorems \ref{t:nt} and \ref{t:nt2} make precise is that a mapping class is reducible if and only if the associated action of $\psi$ on $H(\Sigma)$ has (more or less) a trivial finite-type subrepresentation or if there exists fixed subrepresentation of every finite characteristic quotient of $\pi_1(\Sigma)$.  It is natural to ask whether the following strengthened version of Theorem \ref{t:nt} is true: $\psi$ is pseudo-Anosov if and only if there is a finite cover of $\Sigma$ for which $\psi$ together with the deck transformation group act irreducibly on the rational homology.  This is not true in a very strong sense:

\begin{thm}\label{t:reducible}
Let $\psi\in\Mod^1(\Sigma)$ and $\Sigma'\to\Sigma$ a finite regular cover with deck group $\gam$ which is $\psi$-invariant.  Suppose that $\psi$ acts irreducibly on $H_1(\Sigma',F)$, where $F$ is a field of characteristic zero.  Then $\gam$ is trivial.  Furthermore, there exists a power $n>0$ such that if we replace $\psi$ by $\psi^n$, then there is at least one irreducible $(\psi^n,\gam)$-module for each irreducible $\gam$-module over $F$.
\end{thm}

In particular, every pseudo-Anosov homeomorphism acts reducibly on every nontrivial finite cover of $\Sigma$, and any particular infinite order mapping class can be replaced by a power which acts with a prescribed number of invariant subspaces on some finite cover.  We remark that Theorem \ref{t:reducible} could easily be deduced from the material in \cite{KS}.  Theorem \ref{t:reducible} should be thought of as a dramatic failure of a straightforward generalization of the Casson-Bleiler criterion.

\section{Acknowledgements}
The idea for this paper came out of numerous discussions with Curt McMullen, and the author thanks him for his patience and ceaseless help.  The author thanks Joan Birman for her patience and help with various incarnations of this paper.  The author also thanks Nir Avni, Benson Farb, Stefan Friedl, Eriko Hironaka, Dan Margalit, and Andrew Putman for various helpful conversations.  The author especially thanks Aaron Silberstein for extremely inspiring  and helpful conversations on this topic.  The author thanks Alexandra Pettet and Juan Souto for inspiring the material contained in the Section \ref{s:free}.  Finally, the author thanks the anonymous referees for the time they spent thoroughly understanding this work, and for their numerous helpful, insightful and thorough comments and suggestions.  The author is partially supported by a NSF Graduate Research Fellowship.

\section{Generalities on residual finiteness and automorphisms}
In this section, let $G$ be a finitely generated residually finite group.  The following is well-known and originally due to Baumslag (cf. \cite{LySch}):
\begin{lemma}
$\Aut(G)$ is residually finite.
\end{lemma}
\begin{proof}
If $H<G$ is a finite index subgroup, then $G$ admits a finite index characteristic subgroup contained in $H$.  This claim uses the fact that $G$ is finitely generated.  If $\al\in\Aut(G)$ is nontrivial, then there is a $g\in G$ such that $g\neq\al(g)$.  Since $G$ is residually finite, there is a finite index subgroup $H$ of $G$ which does not contain $g$ and $\al(g)g^{-1}$.  There is a finite index characteristic subgroup $H'<H<G$, and $\al$ descends to an automorphism of $G/H'$.  Since $g$ and $\al(g)g^{-1}$ are both nonidentity elements of $G/H'$, we have that $\al(g)\neq g$ in $G/H'$.  It follows that $\al$ is nontrivial in a finite quotient of $\Aut(G)$.
\end{proof}

It is not true that if $G$ is residually finite then $\Out(G)$ is necessarily residually finite.  By the work of Wise in \cite{W}, any finitely generated group is a subgroup of $\Out(G)$ for some group $G$.

We will need the following observation to prove the homological characterizations of the Nielsen-Thurston classification (see \cite{G} for a similar argument):
\begin{lemma}
Let $G$ be a finitely generated residually finite group.  The following three statements are equivalent:
\begin{enumerate}
\item
$\Out(G)$ is residually finite.
\item
The closure of the image of $G$ in $\widehat{\Aut(G)}$ is separated from all non-inner automorphisms.
\item
Any non-inner automorphism of $G$ descends to a non-inner automorphism of some finite quotient of $G$.
\end{enumerate}
\end{lemma}
\begin{proof}
Since $G$ is residually finite, we have that $\Aut(G)$ injects into its profinite completion $\widehat{\Aut(G)}$.  Strictly speaking, this is true if $G$ has no center, which is certainly the case when $G$ is a free or surface group.  Otherwise, we need to quotient out by the center of $G$ (since the center of $G$ is viewed as a group of trivial inner automorphisms).  To say that $\al$ coincides with an inner automorphism on every characteristic quotient of $G$ is simply to say that $\al$ is an accumulation point of the image of $G$ under the composition \[G\to\Aut(G)\to\widehat{\Aut(G)}.\]  Thus, $\al$ cannot be separated from all inner automorphisms on any finite quotient of $G$ if and only if the closure of image of $G$ in $\widehat{\Aut(G)}$ contains $\al$.  Taking the topological quotient of $\widehat{\Aut(G)}$ by the closure of the image of $G$ gives a Hausdorff space $X$.  It is evident that $X$ is the profinite completion of $\Out(G)$ with respect to a cofinal sequence of finite index subgroups.  $\Out(G)$ is residually finite if and only if $\Out(G)$ injects into its profinite completion, and this happens only if the profinite topology on $\Out(G)$ is Hausdorff.  Note that cosets of finite index subgroups of $\Aut(G)$ form a basis for the profinite topology on $\Aut(G)$.  Therefore $\al$ will be separated from the identity in the profinite completion of $\Out(G)$ if and only if some coset of a finite index subgroup of $\Aut(G)$ separates $\al$ from the image of $G$.  The claim follows.
\end{proof}

It is well-known that both $\Out(F_n)$ and $\Mod(\Sigma)$ are residually finite (see \cite{G}, also \cite{FM}).  The fact that surface and free groups are residually finite can be found in \cite{LySch}, for instance.

\begin{cor}
Let $\psi$ be an automorphism of a surface group or a free group $G$.  Then there is a finite characteristic quotient $\gam$ of $G$ such that $\psi\in\Aut(\gam)$ does not coincide with any inner automorphism of $\gam$.
\end{cor}

Let $G$ be a group which is residually $\mathcal{K}$.  We say that $G$ is {\bf conjugacy separable} with respect to the class of $\mathcal{K}$--groups if for every pair $g,h\in G$ of representatives from different conjugacy classes, there is a $\mathcal{K}$--quotient $K_{g,h}$ of $G$ such that the images of $g$ and $h$ are not conjugate.  We will need the following well-known result:

\begin{lemma}[cf. \cite{LySch}]
Let $G$ be a free or a surface group.  Then $G$ is conjugacy separable with respect to the class of finite $p$-groups.
\end{lemma}

\section{A characteristic zero asymptotically faithful representation of $\Mod^1(\Sigma)$ and highly reducible actions on virtual homology}
In this section, we shall prove Theorem \ref{t:nontrivial}.  There are several proofs, and we will restrict ourselves to the simplest.

\begin{lemma}[cf. \cite{CW}, \cite{KS}]\label{l:cw}
Let $G$ be a free group of finite rank or a finitely generated surface group, and let $\gam$ be a finite quotient of $G$ with kernel $K$.  Then $\gam$ acts faithfully on $H_1(K,\bZ)$.
\end{lemma}
\begin{proof}
Let $X,Y$ be a $K(G,1)$ and a $K(K,1)$ respectively.  If $\gamma\in\gam$ is nontrivial, then $\gamma$ acts fixed-point freely on $Y$, so the Lefschetz number of $\gamma$ must be zero.  Let $n=|\gam|$.  Homology in degrees zero and two is easy to describe, and $\gamma$ acts trivially on them.  In this way, we obtain the character $\chi$ of the representation of $\gam$ on $H_1(Y,\bZ)$: $\chi(1)=2n(g-1)+2$ or $n(d-1)+1$, depending on whether $X$ is homotopy equivalent to a surface of genus $g$ or a wedge of $d$ circles, and $\chi(\gam)=2$ or $1$ respectively when $\gam\neq 1$.  Since characters uniquely determine the isomorphism class of a representation, it follows that the representation of $\gam$ consists of $(2g-2)$ copies of the regular representation and two copies of the trivial representation (respectively $(d-1)$ and one).
\end{proof}

The proof of Lemma \ref{l:cw} is stronger than what we need for Theorem \ref{t:nontrivial}, but we will use this extra information later.  Lemma \ref{l:cw} also follows from the fact that $\gam$ is a finite group of automorphisms and hence $\gam$ acts nontrivially on the homology of the corresponding cover (cf. \cite{FM}, for instance).

The previous lemma can be seen as a generalization of the following result which was the original motivation for the proof of Theorem \ref{t:nontrivial}:
\begin{lemma}\label{l:nontrivial}
Let $\gamma\in G$.  Then there is a finite $p$-group quotient of $G$ with kernel $K$ such that $0\neq [\gamma]\in H_1(K,\bZ)$.
\end{lemma}
\begin{proof}
Suppose $\gamma$ is homologically trivial.  Then there is a finite $p$-group quotient $P$ of $G$ such that $1\neq \gamma\in Z(P)$.  Consider $\Sigma_{P/Z(P)}$, the cover of the base surface corresponding to $P/Z(P)$.  Since $\pi_1(\Sigma_{P/Z(P)})$ admits $Z(P)$ as a quotient, it follows that $\gamma$ is nontrivial in $\pi_1(\Sigma_{P/Z(P)})^{ab}$, whence the claim.
\end{proof}

If $c\subset\Sigma$ is a simple closed curve, one can explicitly produce covers where the homology class of $c$ is nontrivial.  If $c$ is nonseparating, then its homology class is already nontrivial.  Therefore we may assume that $c$ separates $\Sigma$, and $c$ is determined up to a homeomorphism of $\Sigma$ by the splitting on $H_1(\Sigma,\bZ)$ it determines.  Suppose that $\Sigma$ is closed.  Clearly we may arrange $\Sigma$ so that its ``holes" are linearly ordered, and $c$ lies somewhere between the first and the last hole.  Take two simple closed curves which travel ``through the hole", one for each of the fist and last hole.  There is a finite cover $\Sigma'\to\Sigma$ given by counting the sum of the algebraic intersection numbers with each of these curves modulo $2$.  It is easy to verify that $c$ has two distinct lifts to $\Sigma'$, and that each of them is nonseparating.  It is trivial to generalize this construction so that the covering has degree $n$ for any positive integer $n$.

A similar construction holds when $\Sigma$ is not closed and has a finite set $P$ of punctures.  The modification we need to perform is to do intersection theory in $H_1(\Sigma,P,\bZ)$ and use modular algebraic intersection numbers with cycles in $H_1(\Sigma,P,\bZ)$ to find the desired covers.

Lemma \ref{l:cw} shows that if $K<\pi_1(\Sigma)$ is a finite index normal subgroup and $\rho$ is an irreducible representation of $\gam=\pi_1(\Sigma)/K$ over a field $F$ of characteristic zero of dimension $n_{\rho}$, then $\rho$ occurs in the representation of $\gam$ on $H_1(K,F)$ with multiplicity $(2g-2)n_{\rho}$ when $G$ is not free and $\rho$ is nontrivial, with multiplicity $(d-1)n_{\rho}$ when $G$ is free and $\rho$ is nontrivial, and with multiplicity $\rk G$ when $\rho$ is trivial.  We can use this observation to prove that in general, the action of a mapping class on the homology of a finite cover is highly reducible:

\begin{proof}[Proof of Theorem \ref{t:reducible}]
Notice that there is a finite index subgroup of $H_1(\Sigma,\bZ)$, a basis for which is given by curves which lift to $\Sigma'$.  We can assume that these curves are Poincar\'e dual to a basis for the pullback of $H^1(\Sigma,\bZ)$.  Tensoring with $F$, we get a subspace of $H_1(\Sigma',F)$ which is isomorphic to the pullback of $H^1(\Sigma,F)$ by duality.  This is evidently a proper subspace and is invariant under $\psi$ and $\gam$.  The second of these claims follows since $H_1(\Sigma',F)$ decomposes as a direct sum of irreducible modules corresponding to irreducible characters $\chi$ of $\rho$: \[H_1(\Sigma,F)\cong\bigoplus_{\chi}V_{\chi}.\]  The identified subspace which is dual to the pullback of the cohomology of the base corresponds to the trivial representation.

The action of $\psi$ twists the representations of $\gam$, so that $\rho_{\psi}(\gamma):=\rho(\psi(\gamma))$ may not be equal to $\rho$.  If $\rho$ is trivial, then the two are obviously equal.  It follows that the isotypic component of trivial character is $\psi$-invariant, and this is the dual of the pullback of the cohomology of the base.

Since $\psi$ is an automorphism of the finite group $\gam$, some power of $\psi$ acts trivially on $\gam$ and hence preserves the isotypic components corresponding to all irreducible characters of $\gam$.  The second claim follows.
\end{proof}

The method of the previous proof is soft and applies to groups in general.  It seems that whenever $\psi$ is an automorphism of a group with at least some finite index $\psi$-invariant subgroups, then the action of $\psi$ on the homology of these subgroups will be highly reducible.

\begin{proof}[Proof of Theorem \ref{t:nontrivial}]
Let $\gam$ be as in the statement of the theorem and suppose the contrary, so that $\psi(d)=d$ for all $d\in H_1(\Sigma_{\gam},\bZ)$.  Choose $d$ which rationally generates a regular representation, which we can assume to be an integral class by replacing it with a multiple if necessary.  If $\gamma\in\gam$, we have \[\gamma\cdot d=\psi(\gamma\cdot d)=\psi(\gamma)\cdot d,\] which is a contradiction since $d$ generates a regular representation.
\end{proof}

An immediate corollary of the proof of Theorem \ref{t:nontrivial} is the following general statement:
\begin{cor}
Let $G$ be a residually finite group and suppose that each $g\in G$ acts nontrivially by conjugation on the homology of some finite index subgroup $G'_g$ of $G$.  Then each $\psi\in\Aut(G)$ acts nontrivially on the homology of some finite index subgroup $G'_{\psi}$ of $G$.
\end{cor}

After some conversations with the author about early drafts of this paper, Stefan Friedl independently found a proof of Theorem \ref{t:nontrivial} which is essentially the same.
The various corollaries to Theorem \ref{t:nontrivial} follow from the fact that $G$ is residually solvable, nilpotent, $p$, etc.

For certain mapping classes one can explicitly exhibit finite covers of $\Sigma$ to which these mapping classes lift and act nontrivially on the integral homology of the cover.  We can do this explicitly for any Dehn twist, and the idea is identical to the lifting of separating curves to separating ones:

\begin{prop}
Let $p$ be a prime and $c\subset\Sigma$ a simple closed curve on $\Sigma$.  Let $T_c$ denote a Dehn twist about $c$.  Then there is a degree $p$ cover $\Sigma_p$ of $\Sigma$ to which $T_c$ lifts and acts nontrivially on $H_1(\Sigma_p,\bZ)$ when $p\neq 2$, and a degree $4$ cover with this property otherwise.
\end{prop}
\begin{proof}
We may evidently assume that $c$ is separating.  Construct a cover $\Sigma_p$ to which $c$ lifts to $p$ nonseparating curves.  It is clear that when $p>2$ that the simultaneous lifted Dehn twist about these curves acts nontrivially on the homology of $\Sigma_p$.  This can be easily seen by constructing a nonseparating curve which intersects only $2$ of the lifts of $c$.

When $p=2$, $\Sigma_p$ will not do the trick since then $T_c$ lifts to a twist about a bounding pair, which is still an element of the Torelli group.  By taking a $4$-fold cover, we find $T_c$ lifts out of the Torelli group.
\end{proof}

It is conceivable that the methods of the previous proposition could be used to give another more topological proof of Theorem \ref{t:nontrivial}, but there seem to be many technical difficulties arising due to twists about single curves lifting to twists about multicurves.  We can easily see the following, however:

\begin{cor}
Let $\psi\in\Mod(\Sigma)$ be a Dehn twist about a simple closed curve, or more generally a product of Dehn twists about a multicurve in $\Sigma$.  Then $\psi$ acts with infinite order on the homology of a finite cover of $\Sigma$.  We may assume that the cover is a $p$-cover.  In particular, $\psi$ does not lift to an inner automorphism on a finite cover.
\end{cor}

Recall that if $1\neq\psi\in\Mod(\Sigma)$ (or $\in\Out(F_n)$), then we can lift $\psi$ to an automorphism of $\pi_1(\Sigma)$ (or $F_n)$ and find a finite quotient $\gam$ of $\pi_1(\Sigma)$ ($F_n$) on which $\psi$ acts by a non-inner automorphism.  We thus obtain the extension of Theorem \ref{t:nontrivial} which we mentioned in the introduction:

\begin{cor}
Each homeomorphism of $\Sigma$ which is not isotopic to the identity acts nontrivially on the homology of some finite cover $\Sigma'$ of $\Sigma$.  Furthermore, we may assume that the action does not coincide with the action of any element of the deck group of the covering.
\end{cor}
\begin{proof}
There is a finite characteristic quotient $\gam$ of $\pi_1(\Sigma)$ which has the property that $\psi(g)$ is not conjugate to $g$ for any $g\in\gam$.  It follows that there is a $d\in H_1(\Sigma,\bZ)$ such that $\psi(g\cdot d)\neq (h^{-1}gh)\cdot d$ for any $h\in\gam$.
\end{proof}

The exact same argument holds for non-inner automorphisms of free groups.

\section{Encoding closed curves on $\Sigma$ in $H(\Sigma)$}\label{s:encode}
In this section we describe two natural ways of taking an essential closed curve $\gamma\subset\Sigma$ and producing a unique piece of data in $H(\Sigma)$.  In this sense, we shall encode the set of curves in $H(\Sigma)$.

The first method produces so-called {\bf finite type vectors}.  We will associate to each element $\gamma$ of $\pi_1(\Sigma)$ a unique vector $v_{\gamma}\in H(\Sigma)$.  Since elements of $\pi_1(\Sigma)$ are in bijective correspondence with based homotopy classes of loops in $\Sigma$, conjugacy classes of elements of $\pi_1(\Sigma)$ will correspond to free homotopy classes of loops in $\Sigma$.  The set of finite type vectors in $H(\Sigma)$ will have a natural action of $\pi_1(\Sigma)$, and as such that $\pi_1(\Sigma)$--orbits of finite type vectors will be in bijective correspondence with free homotopy classes of essential closed curves in $\Sigma$.  This method associates to each based curve a vector, but unfortunately the span of the finite type vectors is not a subrepresentation of $H(\Sigma)$ in any natural way.

The second method is more representation-theoretic.  To each finite characteristic cover $\Sigma'\to\Sigma$ with deck group $\gam$ we will consider the $\bC\gam$--module $H(\Sigma',\bC)$.  Each irreducible representation of $\gam$ occurs as a direct summand of $H(\Sigma',\bC)$.  Let $\gamma\subset\Sigma$ be an essential closed curve.  On $\Sigma'$, there are two possibilities.  Either $\gamma$ lifts to a union of closed curves or it does not.  The first case occurs exactly when we represent $\gamma$ as a based homotopy class of loops and have that $\gamma$ is contained in the kernel of the map $\pi_1(\Sigma)\to\gam$.  When $\gamma$ lifts to a curve, then each lift of $\gamma$ represents a homology class.  Now suppose that $\gamma$ does not lift.  Representing $\gamma$ by a based loop, we get an action of $\gamma$ on each irreducible representation of $\gam$ in $H_1(\Sigma',\bC)$ or $H^1(\Sigma',\bC)$.  Taking the trace of the action, we get a complex number which does not depend on the choice of lift.  Thus to a based homotopy class of loops $\gamma$ we associate the values of $\chi(\gamma)$, where $\chi$ ranges over all irreducible characters of characteristic quotients of $\pi_1(\Sigma)$.

We will now be more explicit about these constructions.
Let $1\neq \gamma\in \pi_1(\Sigma)$ and suppose that $\gam$ is an abelian quotient of $\pi_1(\Sigma)$.  Assume furthermore that $\gamma$ is non-peripheral.  We can write $\gamma=[\gamma]c$, where $[\gamma]$ is the homology class of $\gamma$ and $c\in [\pi_1(\Sigma),\pi_1(\Sigma)]$.  We identify $[\gamma]$ with a cohomology class which we also call $[\gamma]$ via intersection number (taken relative to the punctures if $\Sigma$ is not closed).  Now consider $H^1(\Sigma',\bQ)$.  We see that algebraic intersection number with $c$ identifies $c$ with a cohomology class $[c]$ in $H^1(\Sigma',\bQ)$, which projects trivially onto the pullback of the cohomology of $\Sigma$.  The finite type vector in $H^1(\Sigma',\bQ)$ should be the vector \[ [\gamma]+[c],\] where we have written \[H^1(\Sigma',\bQ)=H^1(\Sigma,\bQ)\oplus V\] for some $V$.  Without making some choices, this expression is not unique.  We shall soon make precise the choices that need to be made to make the expression canonical.

We are now in a position to define the map $\iota$ for the profinite solvable rational cohomology of $\Sigma$ when $\Sigma$ is closed.  When $\Sigma$ is not closed, we need to consider cohomology relative to the punctures.  We will associate to $g\in\pi_1(\Sigma)$ a vector $v_g\in H(\Sigma)$.  The strategy is to take an element $g\in\pi_1(\Sigma)$ and to look at its ``components" in each $S_i=\pi_1(\Sigma)/D_i(\pi_1(\Sigma))$, the universal $i$-step solvable quotients of $\pi_1(\Sigma)$.  Here $D_i(\pi_1(\Sigma))$ is the $i^{th}$ term of the derived series of $\pi_1(\Sigma)$: \[D_0(\pi_1(\Sigma))=\pi_1(\Sigma),\] \[D_n(\pi_1(\Sigma))=[D_{n-1}(\pi_1(\Sigma)),D_{n-1}(\pi_1(\Sigma))].\]  Comparing the components of $g$ in $S_i$ and $S_{i+1}$ gives us an element in \[D_i(\pi_1(\Sigma))/D_{i+1}(\pi_1(\Sigma)),\] and therefore a homology class of $D_i(\pi_1(\Sigma))$.

First, choose a set of coset representatives for every finite solvable quotient of $\pi_1(\Sigma)$.  It is possible to arrange the choice compatibly, so that if $T_{\gam}$ is a transversal for $\gam$ and $\gam'$ is a quotient of $\gam$, then $T_{\gam'}$ is identified with a subset of $T_{\gam}$.  We thus obtain a set of coset representatives for $D_i(\pi_1(\Sigma))$ for all $i$ as well, since $D_i(\pi_1(\Sigma))$ is the intersection of all finite index subgroups of $\pi_1(\Sigma)$ which contain $D_i(\pi_1(\Sigma))$.  An example of such a choice of coset representatives is to take a basis for $D_i(\pi_1(\Sigma))/D_{i+1}(\pi_1(\Sigma))$ for all $i$, which we then pull back to $\pi_1(\Sigma)$ in some way.  For each finite quotient $F$ of $D_i(\pi_1(\Sigma))/D_{i+1}(\pi_1(\Sigma))$, simply choose elements of $D_i(\pi_1(\Sigma))/D_{i+1}(\pi_1(\Sigma))$ which map onto each $f\in F$.

Let $\gam$ be a finite solvable quotient of $\pi_1(\Sigma)$.  We filter $\gam$ by its derived series to get \[\gam=\gam_0>\gam_1>\cdots>\gam_n=\{1\}.\]  Each quotient $Q_i=\gam/\gam_i$ of $\gam$ gives us a cover $\Sigma_{Q_i}$ of $\Sigma$, and we obtain a tower of covers \[\Sigma_{\gam}\to \Sigma_{Q_{n-1}}\to\cdots\to\Sigma_{Q_1}\to\Sigma,\] where each of the successive deck groups is abelian.  The rational cohomology of $\Sigma_{Q_{i+1}}$ consists of the pullback of $H^1(\Sigma_{Q_i},\bQ)$, together with some new cohomology equipped with a nontrivial action of the abelian group $\gam_i/\gam_{i+1}$.

Let $g\in\pi_1(\Sigma)$.  The homology class of $g$ is $[\gamma]$, which we represent by a pre-chosen coset representative $\gamma$ of $\pi_1(\Sigma)$.  The choice of $\gamma$ is canonical after a choice of coset representatives.  Then $c_1=\gamma^{-1}g\in D_1(\pi_1(\Sigma))$, the first term of the derived series of $\pi_1(\Sigma)$.  Suppose $c_1$ is nontrivial in $\gam$.  Then there is a smallest $i$ such that $c_1$ becomes a homology class in $H_1(\Sigma_{Q_i},\bQ)$, and we record the Poincar\'e dual of this homology class as the entry of $\iota(g)$ corresponding to the cover $\Sigma_{Q_i}\to\Sigma$.

In general, suppose that in $\pi_1(\Sigma)/D_{n+1}(\Sigma)$ we have a representation of the image $g=\gamma\cdot c_1\cdots c_n$, where each $c_i\in D_i(\pi_1(\Sigma))$.  Again, the $c_i$ are canonical after a choice of coset representatives.  If $\gam$ is any finite $n$-step solvable quotient of $\pi_1(\Sigma)$, we define the entries of $\iota(g)$ for all the intermediate covers coming from quotients of $\gam$ as follows.  We take such a quotient $Q$, look at the longest terminal segment $c_i\cdots c_n$ which is trivial in $Q$, and record the homology class of $c_i\cdots c_n$ in $H_1(\Sigma,\bQ)$.  Dualizing, we get a cohomology class, and thus a definition of $\iota$.

When $\Sigma$ has a finite set $P$ of punctures, we need to look at homology and cohomology relative to the punctures to be able to get a duality isomorphism.  This is because Poincar\'e duality gives an isomorphism $H_1(\Sigma,\bZ)\to H^1(\Sigma,P,\bZ)$.

To summarize, $\iota$ is a map $\pi_1(\Sigma)\to H(\Sigma)$.  $H(\Sigma)$ can be thought of as infinite tuples whose entries are classes in $H^1(\Sigma',\bQ)$ for some characteristic solvable cover $\Sigma'\to \Sigma$ with deck group $\gam$.  We assume that $\iota$ has been defined for every characteristic cover lying between $\Sigma'$ and $\Sigma$.  Let $g\in\pi_1(\Sigma)$.  Since $\gam$ is solvable, we have specified a unique coset representative for $g\in\gam$, so that $g=t\cdot c$, where $c\in\pi_1(\Sigma')$.  The entry of $\iota(g)$ corresponding to $\Sigma'$ is the cohomology class given by algebraic intersection number with $c$.

Finally, we call a finite-type vector $v_g$ {\bf peripheral} if $g$ is in the free homotopy class of a small loop about a puncture of $\Sigma$.  The vector $v_g$ is non-peripheral if it is not peripheral.

\begin{lemma}
The map $\iota$ is well-defined.
\end{lemma}
\begin{proof}
This is an immediate consequence of the fact that we chose a compatible set of coset representatives for all finite solvable quotients of $\pi_1(\Sigma)$.
\end{proof}

The definition of $\iota$ shows that the image of $\pi_1(\Sigma)$ under $\iota$ is invariant under the action of $\whG$, the profinite completion of $\pi_1(\Sigma)$.

\begin{lemma}\label{l:injective}
The map $\iota:\pi_1(\Sigma)\to H(\Sigma)$ is injective.
\end{lemma}
\begin{proof}
Since $H(\Sigma)$ is a vector space, $\iota$ cannot a homomorphism if it is to be injective, so it is not sufficient to show that no $\gamma\in \pi_1(\Sigma)$ is trivial under $\iota$.  It is clear, however that if $1\neq\gamma\in \pi_1(\Sigma)$, then $\iota(\gamma)$ is nontrivial in light of Lemma \ref{l:nontrivial}.  Let $\gamma_1,\gamma_2$ be distinct homotopy classes of curves in $\Sigma$.  It follows that for some $i$, the expansions $([\gamma_1],c_{1,1},c_{2,1},\ldots,c_{i,1})$ and $([\gamma_2],c_{1,2},c_{2,2},\ldots,c_{i,2})$ for $\gamma_1$ and $\gamma_2$ in $\pi_1(\Sigma)/D_{i+1}(\pi_1(\Sigma))$ must differ because $\pi_1(\Sigma)$ is residually solvable.  Since $\pi_1(\Sigma)$ is residually $p$-by-torsion-free abelian, there is a finite $p$-cover of $\Sigma$ such that the cohomology classes dual to $c_{i,1}$ and $c_{i,2}$ are different, so that $\iota(\gamma_1)\neq\iota(\gamma_2)$.
\end{proof}

We call the image of $\pi_1(\Sigma)$ the finite type vectors in $H(\Sigma)$.  As we remarked in the introduction, free homotopy classes in $\Sigma$ are in bijective correspondence with conjugacy classes in $\pi_1(\Sigma)$, and hence with $\pi_1(\Sigma)$--orbits of finite type vectors in $H(\Sigma)$.  For the sake of clarity we remark that $\pi_1(\Sigma)$ acts by conjugation on the cohomology of each finite cover of $\Sigma$, and we represent finite type vectors as sums of vectors coming from various representations of the deck group of each cover.  Since these representations are clearly $\pi_1(\Sigma)$--invariant, we see that the conjugation action of $\pi_1(\Sigma)$ on itself and $H(\Sigma)$ commutes with the operation of taking finite type vectors.

To make $\psi\in\Mod(\Sigma)$ act on the finite type vectors, we start with $\gamma\subset\Sigma$ an essential closed curve and its image $\psi(\gamma)$.  To both of these curves we have associated $\pi_1(\Sigma)$--orbits of finite type vectors $v_{\gamma}$ and $v_{\psi(\gamma)}$.  The action of $\psi$ on finite type vectors should take the orbit of $v_{\gamma}$ to the orbit of $v_{\psi(\gamma)}$.

Let us now make a few remarks about the second method.  Since $\pi_1(\Sigma)$ is residually $\mathcal{K}$, where $\mathcal{K}$ is the class of finite, solvable, nilpotent, $p$-groups, etc., we have that each $\gamma\in\pi_1(\Sigma)$ is nontrivial in a $\mathcal{K}$--quotient of $\pi_1(\Sigma)$.  In particular for each $\gamma\in\pi_1(\Sigma)$, there is a $\mathcal{K}$--quotient $\gam$ of $\pi_1(\Sigma)$ and an irreducible character $\chi$ of $\gam$ such that $\chi(\gamma)\neq 0$.  Furthermore, the irreducible characters of $\gamma$ over $\bC$ span the vector space of class functions on $\gamma$, so that if $\gamma$ and $\gamma'$ are not conjugate in $\gam$ then there is an irreducible character which separates them.

\section{The Nielsen-Thurston classification of mapping classes and homology}\label{s:nt}
We are now in a position to give the proofs of the two Nielsen-Thurston classifications.
\begin{proof}[Proof of Theorem \ref{t:nt}]
Let $\psi\in\Mod(\Sigma)$ have finite order.  Lifting $\psi$ to $\Aut(\pi_1(\Sigma))$, some power of $\psi$ is an inner automorphism of $\pi_1(\Sigma)$.  Since we have a bijection between free homotopy classes of curves and $\pi_1(\Sigma)$-orbits of finite type vectors, the characterization of finite order mapping classes is immediate.  This can also be seen from the fact that surface groups are conjugacy separable with respect to the class of finite solvable quotients, so that if $g_1,g_2\in\pi_1(\Sigma)$ are not conjugate then they will be non-conjugate in a finite solvable quotient of $\pi_1(\Sigma)$ (in fact the strongest possible form of this statement is true: they will be non-conjugate in a finite $p$-group quotient of $\pi_1(\Sigma)$).

Reducible mapping classes clearly preserve the $\pi_1(\Sigma)$--orbit of a finite type nonperipheral vector.  Suppose conversely that the conjugacy class of a finite type vector is preserved by $\psi$, but that $\psi$ is pseudo-Anosov.  Then $\psi$ preserves the conjugacy class of a based homotopy class of curves in $\Sigma$.  Lift $\psi$ to $\Aut(\pi_1(\Sigma))$.  On the level of elements of $\pi_1(\Sigma)$, if $1\neq g\in \pi_1(\Sigma)$, we have $\ell(\psi^n_*(g))\sim K^n$, where $K$ is the pseudo-Anosov dilatation of $\psi$ and $\ell$ denotes the word length in $\pi_1(\Sigma)$ (see \cite{FLP}).  The idea behind those asymptotics is the fact that there is a natural metric $\ell_{\mu}$ on $\Sigma$ coming from the invariant foliations of $\psi$ if $\psi$ is pseudo-Anosov.  It turns out that this metric is equivalent to the hyperbolic metric $\ell_h$, or precisely that for any nontrivial class of curves $\gamma$, there exist positive constants $c$ and $C$ such that
\[
c\leq \frac{\ell_h(\gamma)}{\ell_{\mu}(\gamma)}\leq C.
\]
This means that the length of the shortest word representing $\psi^n(g)$ within its conjugacy class grows exponentially.  It follows in our case that $\psi$ cannot be pseudo-Anosov.
\end{proof}

We are now ready for the second version of the Nielsen-Thurston classification.  Let $\psi$ be a reducible mapping class which fixes the conjugacy class of an element $c$ in $\pi_1(\Sigma)$.  Then lifting $\psi$ to an automorphism of $\pi_1(\Sigma)$, we may assume that $\psi$ fixes $c$.

\begin{proof}[Proof of Theorem \ref{t:nt2}]
Note that if $\psi$ is inner then $\psi$ preserves all of the representations of the deck group on any cover.  It follows that the action of $\Mod^1(\Sigma)$ on the representations of the deck group of each cover descends to an action of $\Mod(\Sigma)$.  If the conjugacy class of $c\in \pi_1(\Sigma)$ is invariant under $\psi$, then we see from the proof of Theorem \ref{t:nt} that $\psi$ is reducible.  If $\psi(c)$ is conjugate to $c$, then for each character of the deck group $\gam$, we have that $\chi(\psi(c))=\chi(c)$.  Conversely, if $\psi$ does not preserve the conjugacy class of $c$ then this becomes visible on some cover of $\Sigma$ since $\pi_1(\Sigma)$ is conjugacy-separable.  In particular, there will be a finite characteristic quotient $\gam$ of $\pi_1(\Sigma)$ such that the images of $c$ and $\psi(c)$ are not conjugate.  Since the irreducible characters over $\bC$ form a basis for the class functions of $\gam$ and since every irreducible representation of $\gam$ occurs as a summand of $H_1(\Sigma_{\gam},\bC)$, we see that there is a character of $\chi$ such that $\chi(c)$ and $\chi(\psi(c))$ do not coincide.  The characterization of finite-order mapping classes is clear.
\end{proof}

\section{Free groups and detecting the classification of free group automorphisms}\label{s:free}
Analogously to the Nielsen-Thurston classification, there is a classification of free group automorphisms that can be described using the geometry of Outer space.  An (outer) automorphism $\phi$ of the free group on $n$ generators $F_n$ is called {\bf finite order} if it has finite order in $\Aut(F_n)$ ($\Out(F_n)$).  Recall that $\Out(F_n)$ and $\Aut(F_n)$ act on Outer and Auter space respectively, which are defined as simplicial complexes that parametrize isometry classes of graphs and isometry classes of graphs with basepoint respectively.  An automorphism or outer automorphism $\phi$ has finite order if and only if it fixes a point in Auter or Outer space, respectively.  The analogue of a reducible mapping class is a {\bf reducible} automorphism, which is defined as one which fixes a subgraph of some representative graph $\gam$ satisfying $\pi_1(\gam)=F_n$, with the requirement that the subgraph not be a forest.  An automorphism $\phi$ is called {\bf irreducible} if it is not reducible.  For an accessible introduction to the classification, see \cite{Be}.

By analogy to the construction of $H(\Sigma)$, we may take an exhausting inverse system $\mathcal{K}$ of finite $p$-power index subgroups of $F_n$, abelianize the kernels simultaneously, tensor with $\bC$ and take the inverse limit.  Let us call the resulting vector space $H(F_n)$.  We have:

\begin{cor}
The action of $Aut(F_n)$ on $H(F_n)$ is faithful.
\end{cor}

The main result of this section is the following:

\begin{thm}\label{t:main}
The representation $H(F_n)$ with complex coefficients detects irreducible automorphisms and finite order automorphisms.
\end{thm}

We will need to appeal to the following characterization of reducible automorphisms which can be found in \cite{BH}:

\begin{lemma}\label{l:bh}
Let $\phi\in Out(F_n)$.  Then $\phi$ is reducible if and only if there are free factors $F_{n_i}$, $1\leq i\leq k$, $n_1<n$, such that $F_{n_1}*\cdots *F_{n_k}$ is a free factor of $F_n$ and $\phi$ cyclically permutes the conjugacy classes of the $F_{n_i}$'s.
\end{lemma}

Let $\psi$ be a reducible automorphism of $F$.  By Lemma \ref{l:bh}, we may assume that there is a free factor decomposition $A*B$ of $F$ such that the conjugacy class of $A$ is preserved by $\psi$.  If $a\in A$ is any particular element, we may lift $\psi$ to an automorphism of $F$ which sends $a$ to $A$.  If $\psi$ is irreducible, then for any candidate free decomposition of $F=A*B$, we can find an $a\in A$ such that $\psi(a)$ is not conjugate to an element of $A$.  The next lemma shows the finite index subgroups of $F_n$ detect the failure of a free splitting of $F_n$ to be preserved by an automorphism.

\begin{lemma}
Write $F_n=A*B$.  Suppose $x\in F_n$ is not conjugate to any element of $A$.  Then for every $a\in A$ there is a finite quotient of $F_n$ such that image of $x$ is not conjugate to $a$.  We may assume that this quotient is a $p$-group for any prime $p$.
\end{lemma}
\begin{proof}
This follows from the conjugacy separability of the free group (see \cite{LySch}).  The free group is conjugacy separable with respect to the class of finite $p$-groups, whence the second claim.
\end{proof}

\begin{proof}[Proof of Theorem \ref{t:main}]
Let $\psi$ be irreducible and $A*B$ any candidate splitting.  Lift $\psi$ to $\Aut(F_n)$.  Let $F'<F_n$ be any characteristic subgroup with deck group $\gam$.  As before (cf. Lemma \ref{l:cw}), we have that $\gam$ acts on $H_1(F',\bQ)$, and that each irreducible representation of $\gam$ occurs as a direct summand.  Fix $a\in A$ such that $\psi(a)$ is not conjugate to any element of $A$.  For each $a'\in A$, there is a characteristic $p$-group quotient $\gam$ of $F_n$ and an irreducible character $\chi$ of $\gam$ such that $\chi(\psi(a))$ and $\chi(a')$ do not coincide (we are appealing strongly to the fact that we are doing representation theory over $\bC$).  It follows that the representations of the deck groups of all finite $p$-group quotients of $F_n$ on the complex homology of finite index covers witnesses the fact that $\psi$ does not preserve $A*B$.

If $\psi$ has finite order, then replacing $\psi$ by a power and lifting to $\Aut(F_n)$ shows that for each finite quotient $\gam$ of $F_n$, each $x\in \gam$, and each irreducible character $\chi$ of $\gam$, $\chi(\psi(x))=\chi(x)$.  It follows that $\psi$ must be inner.
\end{proof}

\section{Using representations to approximate $K_{\psi}$}
Though we have obtained a representation-theoretic characterization of the Nielsen-Thurston classification and a faithful representation of the mapping class group, the objects in this paper are difficult to work with.  Understanding the finite nilpotent covers of a thrice-punctured sphere is no easier than understanding all two-generated finite nilpotent groups.  Therefore, even the simplest examples present a lot of difficulty in the setup we consider here.

We can say a few things, though.  For instance, consider the braid groups $B_n$, identified with the the mapping class groups of $n$-times punctured disks.  There is a natural homological representation to consider here: the Burau representation.  Indeed, we have a homomorphism from $F_n\to\bZ$ that takes a word in a fixed standard generating set for $F_n$ to its exponent sum.  This homomorphism gives rise to a covering space $X$ called the {\bf Burau cover}, and by taking values in $\bZ/m\bZ$ we get the {\bf Burau cover modulo $m$}.  The modulo $m$ Burau covers together form the finite Burau covers.  These obviously correspond to covers where small loops about the punctures are all unwound the same number (namely $m$) of times.  The braid group acts on $F_n$ preserving this homomorphism.  We thus get a representation of the braid group on the covering corresponding to the kernel of this homomorphism, which is the classical Burau representation, $V_n$.  It is well known that $V_3$ is faithful and that $V_n$ is not faithful for $n\geq 5$ (see \cite{B2} and the references therein).  The moment $V_n$ is not faithful, there are pseudo-Anosov mapping classes whose nontriviality is not detected by the Burau representation (cf. \cite{I}), much less their dilatations.  Conversely, a representation that contains no pseudo-Anosov classes in its kernel is faithful.  Since the kernel and image of the homomorphism are $B_n$ invariant, we see that if $X$ is the cover of $\Sigma$ corresponding to the kernel of the homomorphism, $B_n$ acts on $H_1(X,\bQ)$ by $\bZ[t^{\pm1}]$--linear maps.  We can set the parameter to be a root of unity and thus obtain a representation of the braid group on the homology of the multiply punctured disk with twisted coefficients.  Such homology groups arise naturally from covering spaces.  Letting $t$ be a primitive $n^{th}$ root of unity gives rise to the twisted homology given by the Burau cover modulo $n$.  It is easy to see the following proposition, whose phrasing was first suggested to the author by McMullen:

\begin{prop}[cf. \cite{BB}]\label{t:bsup}
Let $\psi\in B_n$ be a pseudo-Anosov braid and $S^1$ denote the unit complex numbers.  Then $\sup_{t\in S^1}\rho(V_n(\psi))\leq K$, and the supremum represents finite cyclic covers of $\Sigma=\Sigma_{0,n,1}$ that have equal branching over each of the punctures.
\end{prop}

The first claim in the proposition follows from general principles about Lipschitz maps acting on metric manifolds which we stated above as Theorem \ref{t:bound} together with the observation that there is a metric for which a pseudo-Anosov map with dilatation $K$ is $K$--Lipschitz.  Though multiply punctured surfaces are not compact, there are compact manifolds with boundary which are deformation retracts of multiply punctured surfaces, and thus Theorem \ref{t:bound} holds for surfaces with punctures by applying an appropriate limiting process.  Precisely, we can remove $\eps/2$--neighborhoods of the punctures to get a compact manifold with boundary, and we can find a new homeomorphism on the truncated surface which coincides with a $K$--Lipschitz homeomorphism $f$ outside of an $\eps$-neighborhood of the punctures.  We then let $\eps\to 0$.  It will be clear from the argument below that even considering surfaces with punctures will not introduce significant technical obstacles.

\begin{proof}[Proof of Theorem \ref{t:bound}]
The proof proceeds by producing a bound for simplicial chains, and estimating the growth of the action of the homeomorphism $f^n$ on the vector space of real simplicial chains.  We can then estimate the spectral radius $\rho$ of the action of $f$ on the chains and then deduce an estimate of the spectral radius of the induced action on the real homology.

Choose a simplicial decomposition of $X$ with a very fine simplices, so that any point in $X$ is $\ll1/K$ from the barycenter of a simplex.  Then, if $\gamma$ is a loop in $X$, we will be able to homotope it to a path lying in the $1$-skeleton of $X$ without increasing the length by more than some constant factor $C$ that will work for all of $X$.  Choose the subdivision also so that all the $1$--simplices have approximately the same size.  This can be done as follows: consider the length spectrum of all $1$--simplices for some decomposition.  This is a finite set since $X$ is compact.  Consider any two positive real numbers $s,t$.  For any $\eps>0$, there exist integers $m,n$ such that $|s/m-t/n|<\eps$.  By an easy induction, for any $\eps$ and any finite collection of positive real numbers, we can find a sequence of integers such that the quotients differ pairwise by no more than $\eps$.

Therefore, given any $\eps$, we can subdivide the $1$-skeleton of our simplicial complex so that the lengths of any two $1$-simplices differ by no more than $\eps$.  Now let $z$ be a $1$--cycle.  Writing $z$ as a weighted union of $1$-simplices, we may talk about the length $\ell$ of $z$.  Let $f_*$ denote the linear map on real simplicial chains induced by $f$, and consider $f^n_*(z)$ and $f^n(z)$.  Since $f$ is $K$--Lipschitz, the length of $f^n(z)$ is no more than $K^n\cdot\ell$.  On the other hand we can homotope $f^n(z)$ to the $1$--skeleton, thus increasing its length to no more than $C\cdot K^n\cdot\ell$.
Choosing $\eps$ small enough, we see that if $z$ required $m$ $1$--simplices to be expressed as a $1$-chain, $f^n_*(z)$ requires no more than $C/(1-\eps)\cdot K^n\cdot m$ simplices.

Let $G$ be a finitely generated group, and $g:G\to H$ a surjective homomorphism.  There is a well-defined length function on $G$ given by the graph metric on a fixed Cayley graph for $G$, and it induces a length function on $H$.  Furthermore, it is clear that the induced length function is bounded by the length function on $\ell_G$, i.e. $\ell_G(\gamma)\geq\ell_H(g(\gamma))$ for all $\gamma\in G$.  Since homology is a quotient of a subgroup of the $1$--chains, we obtain $\rho_1(f_*)\leq K$.

The proof in general is analogous.  The metric gives us a way to measure the volume of simplices: the volume of a very small $m$-cube is given by the $m^{th}$ power of the side lengths.  Therefore, $f$ will scale the volume element by no more than $K^m$.  As before, we can cut up $m$-simplices so that their volumes are all similar.  So, if $z$ is an $m$-cycle, we can homotope $f^n(z)$ to sit in the $m$-skeleton of $X$, increasing its volume by no more than a factor of some constant $C$.  The constant $C$ can be estimated as follows: if an $m$-chain $c$ intersects the interior of an $m'$-simplex $S$ with $m'>m$, then we perform a homotopy $rel\,\partial S$ to push $c$ to the $m'-1$-skeleton $S\cap X_{m'-1}$ and proceed inductively.  Subdividing the interior of $S$ if necessary, we can assume the homotopy does not change the volume of $c\cap S$ by much.  The subdivision of the $m$-skeleton into simplices of similar size can be done afterwards without altering the validity of the proof.  This proves the claim.
\end{proof}
The second claim in proposition \ref{t:bsup}, though intuitively clear, needs some argument, and a complete proof can be found in \cite{BB}.

Note that there is a two-sheeted covering of the thrice punctured disk that gives a thrice punctured torus with one boundary component.  Furthermore, any simple pole of any quadratic differential is resolved, so that we get a quadratic differential on a once-punctured torus.  To see that poles of quadratic differentials can generally be resolved to be regular points on a finite cover we have the following result:

\begin{prop}\label{l:monogon}
Suppose that $\psi$ stabilizes a Teichm\"uller geodesic determined by a quadratic differential $q=q(z)\, dz^2$ which has only simple poles and even order zeros.  Then there is a finite unbranched cover $\Sigma'$ of $\Sigma$ such that the lift of $q$ has only even order zeros.  In particular, there exists a finite unbranched cover $\Sigma'\to\Sigma$ such that for any lift of $\psi$ to $\Sigma'$, the homological and geometric dilatations of $\psi$ coincide.
\end{prop}
\begin{proof}
Let $P$ be the set of punctures of $\Sigma$.  At each point in $P$ we may assume the stable foliation $\mathcal{F}$ determined by $q$ has either an even-pronged singularity or a one-pronged singularity.  In the latter case, we may simply fill in the missing point since the quadratic differential associated to $\mathcal{F}$ has a zero at that point (so that the puncture is a removable singularity).  Therefore we may assume that all the points in $P$ are poles of the associated quadratic differential.  By passing to a characteristic cover of $\Sigma$, we may assume $|P|$ is even.  Such a cover might be given by taking the cover associated to $H_1(\Sigma,\bZ/2\bZ)$ if $|P|=1$.  If $|P|>1$, label the punctures $p_1,\ldots,p_n$.  Take a small loop about each puncture and record its homology class.  Send the homology classes of the loops about $p_1,\ldots,p_{n-1}$ to $1\in\bZ/2\bZ$.  Since the sum of the the homology classes of these loops must be zero, if $n$ is even we may send the homology class of the loop about $p_n$ to $1\in\bZ/2\bZ$.  Otherwise, we are forced to send it to $0\in\bZ/2\bZ$.  Since the first cover is characteristic and since punctures of $\Sigma$ are $\psi$-invariant, $\psi$ lifts to the covers determined by these homomorphisms.  After taking this cover, it is obvious that all but possibly one puncture in $P$ will lift to a regular point of the foliation.  Since the cover has even degree, the one puncture over which the cover did not ``branch" lifts to an even number of punctures.  Repeating the second cover construction, we construct a further unbranched cover (which can clearly be refined to a $\psi$-invariant or even characteristic cover, since these punctures are poles of the quadratic differential associated to $\mathcal{F}$) where a small loop about each puncture is unwound to half a loop about one of the punctures in the cover.

Once we have unwound each puncture as above, each pole of the quadratic differential becomes a regular point and hence can be filled in as a removable singularity.  It follows that after passing to another finite cover if necessary, $q$ becomes the global square of a holomorphic $1$-form $\omega$, whence the claim.
\end{proof}

We may view $B_3/Z(B_3)$ (where the center is generated by a twist about the boundary of the disk) as a subgroup of the mapping class group of the once-punctured torus.  The homological representation theory of this group is well-understood, especially in connection to the Nielsen-Thurston classification.  The situation is already much more complicated for the four-times punctured disk.  Let $\{\sigma_1,\ldots,\sigma_{n-1}\}$ denote the standard generators of the braid group $B_n$.  Then $\beta=\sigma_1\sigma_2\sigma_3^{-1}\in B_4$ is pseudo-Anosov (see \cite{HK} for a large class of examples in this same flavor.)  According to Hironaka and Kin the dilatation of $\beta$ is the largest root of the polynomial
\[
1-t-2t^2-2t^3-t^4+t^5,
\]
which is approximately $2.29663$.  Applying the machinery of proposition \ref{t:bsup}, we get that the Burau matrix $V_4(\beta)$ is
\[
M=
\begin{pmatrix}
-t&-t^2&-t^2\\1&0&0\\0&1&1-1/t
\end{pmatrix}.
\]
The characteristic polynomial of $M$ is \[t+u-tu+t^2u-u^2+u^2/t+tu^2+u^3.\]  The supremum of the spectral radii of $M=M(t)$ as $t$ varies over $S^1$ is approximately $2.17401$.  By choosing a small mesh size, the estimate will be close to the actual value.  It is possible to show that the inequality between the supremum of the spectral radii of $M(t)$ and $K$ is strict in this case by showing that $2.17401$ is already very close to the actual Burau supremum.  It is a theorem of Band and Boyland in \cite{BB} that the spectral radius of a Burau matrix specialized at a root of unity is either equal to the dilatation when $t=-1$ or is strictly smaller than the dilatation.  Our computation shows that the difference between the spectral radii and the dilatation can be bounded away from zero.  The question of whether $K$ is achieved if the covers are allowed to vary over all covers of $\Sigma$ is a consequence of \cite{Mc}.

The strict inequality of Proposition \ref{t:bsup} does not change if we pass to the Lawrence-Krammer representation, which is a well-known faithful representation of the braid group.  For more detail, see for instance \cite{B}.
Recall that the configuration space of pairs of points in a space $X$ is the set \[(X\times X\setminus\Delta)/(\bZ/2\bZ)),\] where the group action permutes the coordinates.  In the case of a disk punctured at $p_1,\ldots,p_n$, the configuration space $C$ is a $4$-manifold and inherits a natural action of the braid group.  The representation itself is the action of $B_n$ on the second (usual) homology of a certain $\bZ^2$-cover of $C$, viewed as a $\bZ[t^{\pm 1},q^{\pm 1}]$--module.

We now give a verbatim account of the setup in which Bigelow works in \cite{B}.  If $\al$ is a path in $C$, we may view $\al$ as $\{\al_1,\al_2\}$, where we mean unordered pairs of points.  Let
\[
a=\frac{1}{2\pi i}\sum_{j=1}^n\left(\int_{\al_1}\frac{dz}{z-p_j}+\int_{\al_2}\frac{dz}{z-p_j}\right)
\]
and
\[
b=\frac{1}{\pi i}\int_{\al_1-\al_2}\frac{dz}{z}.
\]
These quantities are $B_n$-invariant, so that $B_n$ acts on $H_2(Z,\bZ)$ as a $\bZ[t^{\pm 1},q^{\pm 1}]$--module, where $Z$ is the covering space corresponding to the map $\al\mapsto q^at^b$.

Bigelow provides explicit matrices for the corresponding representation, which allows for relatively simple computation of the supremum of the homological dilatations over finite intermediate covers between $C$ and $Z$ (the proof of this is again analogous to Proposition \ref{t:bsup}).  For the Hironaka-Kin example, we obtain a $6\times 6$ matrix which we do not reproduce here.  Once we obtain the supremum, we must take the square root, since the action is on second homology.  Theorem \ref{t:bound} applies for the Lawrence-Krammer representation, since the action of a homeomorphism on the covering space is inherited from the action on the coordinates.  If the action comes from a pseudo-Anosov homeomorphism with dilatation $K$, those actions can be taken to be $K$--Lipschitz.  Theorem \ref{t:bound} guarantees that this estimate will not exceed the dilatation.  We obtain the supremum $2.17433<2.29663$.

\section{Questions and examples}\label{s:examples}
A basic question which the preceding discussion leaves open is the following:

\begin{quest}
Let $\psi$ be an infinite order mapping class.  Does $\psi$ act with infinite order on the (co)homology of some finite cover?  If $\psi$ has positive entropy, can $\psi$ be made to act with spectral radius greater than one?
\end{quest}

We can formulate analogous questions to this one and Question \ref{q:homdil} for outer automorphisms of free groups.

The first part of this question is really only open for an automorphisms of surfaces which are generalized pseudo-Anosov, in the following sense.  Given $\psi\in\Mod(\Sigma)$, find a canonical reduction system $\mathcal{C}$ for $\psi$.  Then on every component of $\Sigma\setminus\mathcal{C}$, $\psi$ either acts trivially (up to isotopy) or as a pseudo-Anosov homeomorphism, possibly after passing to a power of $\psi$.  If there is a $\mathcal{C}'\subset \mathcal{C}$ such that $\psi$ does a combination of Dehn twists about $\mathcal{C}'$, then we have already shown that $\psi$ will act with infinite order on the homology of a finite cover.  By a Dehn twist about $c\in\mathcal{C}'$, we mean that in $\Sigma\setminus\mathcal{C}$, $c$ forms two boundary components in $\Sigma\setminus\mathcal{C}$, and we require that the corresponding components of $\Sigma\setminus\mathcal{C}$ have trivial restriction of $\psi$.  Here we are using that $\mathcal{C}$ is indeed a canonical reduction system.
Note that $\psi$ has to have infinite order in $\Mod(\Sigma)$, as no inner automorphism can act with infinite order on the homology of a finite cover.

If the answer to this question or the analogous one for free groups is true, then $\psi$ will act with infinite order on $V_{\chi}^{2g-2}\subset H^1(\Sigma_{\gam},\bQ)$ (resp. on the homology of some finite index subgroup of $F_n$ with quotient $\gam$), where $V_{\chi}$ is an irreducible representation of $\gam$ and $V_{\chi}^{2g-2}$ is its isotypic component.  A rather dramatic example of this fact is the following observation:

Recall that $\mathcal{I}(\Sigma)$ admits the Johnson filtration.  It can be easily shown that any element of $J_1$ acts trivially on $\gamma_k(G)/\gamma_{k+1}(G)$ for all $k$ (cf. \cite{BL}).  Let $\psi\in J_k\setminus J_{k+1}$.  It follows that there is an element $g\in G$ and $z\in\gamma_k(G)/\gamma_{k+1}(G)$ such that $\psi(g)=gz$ in $G/\gamma_{k+1}(G)$.  For all $n\neq 0$, we have $\psi(g^n)=g^nz^n$ in that quotient.

Now let $\Sigma'$ be a finite cover of $\Sigma$ where $z$ lifts to a nontrivial homology class and let $n$ be such that $g^n$ lifts to a homology class.  Consider $d=g^n$ as a homology class in $\Sigma'$, and suppose that $\psi(g)=g\cdot z$.  It follows that $\psi(d)$ is equal to $d+ z'$, where $z'$ is the sum of the conjugates of $z$ by all powers of $g$, where $g$ is viewed as a deck transformation of the covering.  The representation of the subgroup $\langle g\rangle$ where the homology class of $z$ is located can be assumed to be a trivial representation if and only if there is a finite quotient $\gam'$ of $G$ such that both $g$ and $z$ are in the kernel of the quotient map, and if both $g$ and $z$ are nontrivial in the abelianization of the kernel.

If the conjugation action is nontrivial, then the sum of the conjugates of $z$ is zero.  Indeed, tensoring with $\bC$ shows that $g^k$ acts by $\zeta_n^k$, where $\zeta_n$ is an $n^{th}$ root of unity, whence the claim.

Finally, one would like to produce a more useful homological version of the Nielsen-Thurston classification.  Recall that there is the so-called Casson-Bleiler criterion for a map to be pseudo-Anosov (see \cite{CB}).  We have seen that a na\"ive generalization of the Casson-Bleiler criterion cannot hold by Theorem \ref{t:reducible}.

\begin{quest}
Is there a characterization of pseudo-Anosov maps which is homologically visible on some finite cover?  Some $p$-power cover?
\end{quest}

\section{Some remarks on the uniqueness of finite type vectors in $H(\Sigma)$}
Though it is disappointing that finite type vectors in $H(\Sigma)$ do not form a subrepresentation, we will now show that they are about as good a construction as one can expect.  If we want to encode homotopy classes of curves as vectors in $H(\Sigma)$, there are certain natural conditions we should put on the image of each homotopy class.  Imposing just a few relatively mild assumptions greatly restricts the possible ways to encode $\pi_1(\Sigma)$.

\begin{prop}
Let $\iota:\pi_1(\Sigma)\to H(\Sigma)$ be an injective map.  Suppose that $\iota$ satisfies the following two conditions:
\begin{enumerate}
\item
Let $\Sigma'\to \Sigma$ be one of the finite covers occurring in the construction of $H(\Sigma)$, and let $\gam$ be the deck group of the cover.  If $g\in\pi_1(\Sigma)$ is in the kernel of $\pi_1(\Sigma)\to\gam$, the entry of $\iota(g)$ corresponding to $\Sigma'$ is the image of $g$ in $H_1(\Sigma',\bZ)$.
\item
Let $\Sigma'\to\Sigma$ be a finite solvable cover.  If $g_1,g_2\in \pi_1(\Sigma')$ are homologous, then $\iota(g_1)-\iota(g_2)=\iota(g_2^{-1}g_1)$ in the entry corresponding to $\Sigma'$.
\end{enumerate}
Then $\iota$ coincides with the $\iota$ defined in Section \ref{s:encode}, which is well-defined up to some choices of coset representatives.
\end{prop}
\begin{proof}
Let $g\in \pi_1(\Sigma)$.  On $\Sigma$, we must first take the homology class of $g$.  By identifying coset representatives for each term of the derived series of $\pi_1(\Sigma)$, we found a canonical element $\gamma$ of $\pi_1(\Sigma)$ which is homologous to $g$.  By condition (2), we found the smallest solvable cover of $\Sigma$ where $\gamma^{-1}g$ becomes a homology class and repeated the construction.
\end{proof}

The representation $H(\Sigma)$ has some rather pathological properties which make the characterization of mapping classes by their action on $H(\Sigma)$ rather subtle.  For instance:
\begin{prop}
Suppose that $\psi\in\Mod(\Sigma)$ fixes a positive dimensional subspace of $H^1(\Sigma,\bQ)$.  After lifting $\psi$ to $\Mod^1(\Sigma)$ arbitrarily, $\psi$ fixes a vector in $H(\Sigma)$.
\end{prop}
\begin{proof}
This is obvious, since we can pullback the cohomology of $\Sigma$ to each cover $\Sigma'$ and extend a vector in $H^1(\Sigma,\bQ)$ to $H^1(\Sigma',\bQ)$ by zero.  Analogously, we can replace $H^1(\Sigma',\bQ)$ by $H_1(\Sigma',\bQ)$, except we then need to rescale by the degree of the cover $\Sigma'\to\Sigma$.
\end{proof}

Any more interesting data to be found in the virtual homological representation of $\Mod^1(\Sigma)$ has to be found in the action of an automorphism on the isotypic components of the representations of a finite characteristic quotient of $\pi_1(\Sigma)$.  Describing these actions seems rather difficult, certainly no easier than understanding the action of an automorphism on a characteristic quotient of $\pi_1(\Sigma)$.  Certain invariants, such as the dilatation of a pseudo-Anosov homeomorphism, may be hidden somewhere in these representations, but it is not at all obvious where.  If $\psi$ is not virtually a homological pseudo-Anosov mapping class for instance (any Torelli mapping class of a closed surface will do), then the homological dilatation of $\psi$ will be strictly smaller than the geometric dilatation of $\psi$ (cf. \cite{BB}, \cite{KS}, \cite{LT}).

\end{document}